\documentclass[11pt,reqno]{amsart}
\usepackage{amssymb,latexsym}
\usepackage{color}
\usepackage{graphicx}
\usepackage[mathcal]{euscript}
\usepackage{mathrsfs}
\usepackage{enumerate}
\usepackage{hyperref}

\newtheorem{thm}{Theorem}[section]

\newtheorem{lem}[thm]{Lemma}
\newtheorem{note}[thm]{Note}
\newtheorem{rem}[thm]{Remark}
\newtheorem{prop}[thm]{Proposition}
\theoremstyle{definition}
\newtheorem{defn}[thm]{Definition}
\numberwithin{equation}{section}

\allowdisplaybreaks
\begin{document}

\setcounter{page}{1}

\title[$k$th-order slant little Hankel operators on the Bergman space]{Commutativity and spectral properties of $k^{th}$ -order slant little Hankel operators on the Bergman space}

\author[{\bf Anuradha Gupta and Bhawna Gupta}]{{\bf Anuradha Gupta \MakeLowercase{and}  Bhawna Gupta}}

\address{Anuradha Gupta, Associate professor, Delhi college of arts and commerce, Department of mathematics, University of Delhi, Delhi, India.}
\email{dishna2@yahoo.in}

\address{Bhawna Gupta, 
Department of mathematics,
University of Delhi,
New Delhi-110007,
India.}
\email{swastik.bhawna26@gmail.com}



\begin{abstract}
In this paper, we introduce the notion of $k^{th}$-order slant little Hankel operator on the Bergman space  with essentially bounded harmonic symbols on the unit disc and obtain its algebraic and  spectral properties. We have also discussed the conditions under which $k^{th}$-order slant little Hankel operators commute.
\end{abstract} \maketitle

\section{{\bf Introduction}}
Let $ \mathbb{D}= \{ z \in \mathbb{C} : |z| <1 \}$ be the open unit disc and $ dA= dxdy $ denotes the Lebesgue area measure  on $ \mathbb{D} $, normalised so that the measure of $ \mathbb{D}$ is $ 1 $. Let  $ L^2(\mathbb{D},dA) $ be the space of all Lebesgue measurable functions  $ f $  on $ \mathbb{D} $ for which $$ \| f \|^2 = \int _\mathbb{D} |f(z)|^2 dA(z) < \infty.$$ It forms a Hilbert space with inner product $$\langle f , g \rangle = \int _\mathbb{D} f(z)\overline{g(z)} dA(z).$$

The Bergman space $ L_a^2(\mathbb{D},dA)$ consists of all analytic functions f  on $ \mathbb{D} $  such that $ f \in L^2(\mathbb{D},dA) $ and is closed subspace of Hilbert space $ L^2(\mathbb{D},dA) $. It is well known that $\{ \sqrt{n+1} z^n\}_{n=0}^\infty $ is an orthonormal basis for $ L_a^2(\mathbb{D},dA)$ and the orthogonal projection $ P:  L^2(\mathbb{D},dA)  \longrightarrow  L_a^2(\mathbb{D},dA) $
is an integral operator
$$Pf(z)=\langle f,K_z \rangle =\int _\mathbb{D} K(z,w)f(w)dA(w),$$
where $ K(z,w)=K_w(z) =\frac{ 1}{(1-\bar{w}z)^2} $ is the unique reproducing kernel of $ L_a^2(\mathbb{D},dA)$. The space $ L^{\infty}(\mathbb{D},dA) $ is the Banach space of Lebesgue measurable functions $ f $ on $ \mathbb{D} $ such that 
$
\| f \|_\infty = ess sup \{ |f(z)| : z \in \mathbb{D} \} < \infty
$ and $ H_\infty (\mathbb{D},dA) $ denotes the set of all analytic functions of the space $ L^{\infty}(\mathbb{D},dA) $.

The theory of Hankel operators on the Hardy spaces is an important area of mathematical analysis and lots of applications in different domains of mathematics have been found such as interpolation problems, rational approximation, stationary processes or pertubation theory \cite{pellerexcursion, pellerhankel}. In the year $ 1996 $, M.C.Ho \cite{hoproperties} investigated the basic properties of slant Toeplitz operators on Hardy spaces. After that in the year $ 2006 $, Arora \cite{aroraslant} introduced the class of slant Hankel operators on Hardy spaces and discussed its characterizations.

In the present paper, we study spectral and commutative properties of $k^{th}$-order slant little Hankel operators on the Bergman space with essentially bounded harmonic symbols. More precisely, we describe the conditions under which $k^{th}$-order slant little Hankel operators commute and we prove that spectrum and approximate point spectrum of $ S_{\phi}^k $ are same where $ \phi(z) = \sum_{i=0}^N \bar{z}^i $ and $ N \in \{ 0,1,\cdots 2k-1 \}$. Basic properties of the Hardy space and Bergman  spaces can be found in \cite{douglasbanach, durenbergman}. We refer  \cite{aroraslant, aroracompression, hoproperties, powerhankel} for the applications and extensions of study to Hankel operators, slant Toeplitz operators, slant Hankel operators and its generalization on Hardy spaces.  Let $ T $ be a bounded linear operator on a complex Banach space $ X $ then spectrum of $ T $ is defined as the set of all complex number $ \lambda $ such that $ \lambda I - T $ is not invertible in the algebra $ \mathbb{B}(X) $, where $ I $ denotes the identity operator on $ X $ and $ \mathbb{B}(X) $ denotes the set of all bounded linear operator on $ X $ . The spectrum of $ T $ is classified into three categories: point spectrum, $ \sigma_P(T) $; residual spectrum, $ \sigma_R(T) $ and continuous spectrum, $ \sigma_C(T) $ where 
\begin{align*}
\sigma_P(T) &= \{ \lambda \in \mathbb{C}: ker(\lambda I - T ) \neq (0)\},\\
\sigma_R(T) &= \{ \lambda \in \mathbb{C}: ker(\lambda I - T ) = (0) \text{ and } Range (\lambda I - T )^- \neq X\},\\
\text{ and }\\
\sigma_C(T) &= \{ \lambda \in \mathbb{C}: ker(\lambda I - T ) = (0) \text{ and } Range (\lambda I - T ) \neq Range (\lambda I - T )^- = X \}.
\end{align*}
 There are some overlapping divisions of spectrum also, namely approximate point spectrum and compression spectrum where approximate point spectrum of $ T $ is the set of all complex number $ \lambda $ such that $ \lambda I - T $ is not bounded below and compression spectrum of $ T $ is the set of all complex number $ \lambda $ such that $ Range(\lambda I - T)^- \neq X $.

\section{{\bf $k^{th}$-order slant little Hankel operators on $L_a^2(\mathbb{D},dA)$}}
Let $ \phi \in L^{\infty}(\mathbb{D},dA) $ then for any $ f \in L_a^2(\mathbb{D},dA) $, the Toeplitz operator  $T_\phi : L_a^2(\mathbb{D},dA) \longrightarrow L_a^2(\mathbb{D},dA) $ is defined as $ T_\phi(f) = P(\phi f) $ and the little Hankel operator  $ H_\phi : L_a^2(\mathbb{D},dA) \longrightarrow  L_a^2(\mathbb{D},dA)$ is defined as $ H_\phi(f) = PJM_\phi(f) $ where P is the orthogonal projection of $ L^2(\mathbb{D},dA) $ onto $ L_a^2(\mathbb{D},dA)$, $ J: L^2(\mathbb{D},dA) \longrightarrow L^2(\mathbb{D},dA) $ is defined by $ J(f(z)) = f(\bar{z}) $ and $M_\phi$ is the multiplication operator on $ L_a^2(\mathbb{D},dA) $ defined as $ M_\phi(f) = \phi f $. It is well known that $H_\phi $ is bounded with $ \|H_\phi\| \leq \|\phi\|_\infty $. For $  \phi = \sum_{j=0}^\infty a_j \bar{z}^j +\sum_{j=1}^\infty b_j z^j  \in L^{\infty}( \mathbb{D},dA) $, the $ (m,n)^{th} $ entry of matrix representation of little Hankel operator on $L_a^2(\mathbb{D},dA)$ with respect to orthonormal basis  $\{\sqrt{n+1} z^n\}_{n=0}^\infty $ is
\begin{align*}
\left\langle H_\phi \sqrt{n+1}z^n,\sqrt{m+1}z^m \right\rangle 
&=  \sqrt{n+1}  \sqrt{m+1} \left\langle PJ( \phi z^n),z^m \right\rangle\\
&= \sqrt{n+1}  \sqrt{m+1} \left\langle (\sum_{j=0}^\infty a_j \bar{z}^j +\sum_{j=1}^\infty b_j z^j ) z^n,\bar{z}^m \right\rangle\\
&= \sqrt{n+1}  \sqrt{m+1}\sum_{j=0}^\infty a_j \left\langle  \bar{z}^j z^n,\bar{z}^m \right\rangle\\
&=\frac{\sqrt{n+1}  \sqrt{m+1}}{ (m+n+1)}a_{m+n}\\
\end{align*}
and its matrix representation is

\begin{equation*}
[H_\phi] =
\begin{bmatrix}
a_0 & \frac{1}{\sqrt{2}}a_1 & \frac{1}{\sqrt{3}}a_2 & \frac{1}{\sqrt{4}}a_3 & \frac{1}{\sqrt{4}}a_4 & \dots \\\\
\frac{1}{ \sqrt{2}}a_1 & \frac{2}{3}a_2 & \frac{\sqrt{6}}{4}a_3 & \frac{\sqrt{8}}{5}a_4 & \frac{\sqrt{10}}{6}a_5 & \dots \\\\
\frac{1}{\sqrt{3}}a_2 & \frac{\sqrt{6}}{4}a_3 & \frac{3}{5}a_4 & \frac{\sqrt{12}}{6}a_5 & \frac{\sqrt{15}}{7}a_6 & \dots \\\\
\vdots & \vdots & \vdots & \vdots & \vdots & \dots \\\\
\end{bmatrix}
\end{equation*}
whose adjoint is given by
\begin{equation*}
[H_\phi]^* =
\begin{bmatrix}
\overline{a_0} & \frac{1}{\sqrt{2}}\overline{a_1} & \frac{1}{\sqrt{3}}\overline{a_2} & \frac{1}{\sqrt{4}}\overline{a_3} & \frac{1}{\sqrt{4}}\overline{a_4} & \dots \\\\
\frac{1}{ \sqrt{2}}\overline{a_1} & \frac{2}{3}\overline{a_2} & \frac{\sqrt{6}}{4}\overline{a_3} & \frac{\sqrt{8}}{5}\overline{a_4} & \frac{\sqrt{10}}{6}\overline{a_5} & \dots \\\\
\frac{1}{\sqrt{3}}\overline{a_2} & \frac{\sqrt{6}}{4}\overline{a_3} & \frac{3}{5}\overline{a_4} & \frac{\sqrt{12}}{6}\overline{a_5} & \frac{\sqrt{15}}{7}\overline{a_6} & \dots \\\\
\vdots & \vdots & \vdots & \vdots & \vdots & \dots \\\\
\end{bmatrix}.
\end{equation*}
 
 \noindent From above matrices, we conclude that $ H_\phi^* = H_{\hat{\phi}} $ where $ \hat{ \phi}(z) = \sum_{j=0}^\infty \overline{a_j} \bar{z}^j +\sum_{j=1}^\infty \overline{ b_j} z^j  \in L^\infty( \mathbb{D}, dA)$.\\
For $ k \geq 2 $, define $ W_k : L_a^2(\mathbb{D},dA) \longrightarrow L_a^2(\mathbb{D},dA)$ by $ W_k(z^{kn}) = z^n, W_k(z^{kn+p}) = 0 $ for all $ n \in \mathbb{N} \cup \{ 0\} $ and $  p=1,2,\ldots {k-1} $. Clearly, $W_k$ is a bounded linear operator with  $ \| W_k \| = \sqrt{k} $ and the  adjoint of $ W_k $ is given by $ W_k^*(z^m) =\frac{km+1}{m+1}z^{km}  \text { for } m \geq 0 $ (see \cite{lucommutativity}).
 
 In year $ 2008 $, Arora and Bhola \cite{aroracompression} discussed about the $k^{th}$-order slant Hankel operators on $ H^2 $ space. We extend the definition of little Hankel operators on the Bergman space to $k^{th}$-order slant little Hankel operators in the following manner: 
\begin{defn}
For $ k \geq 2 $ and $ \phi(z) = \sum_{j=0}^\infty a_j \bar{z}^j +\sum_{j=1}^\infty b_j z^j $ in $ L^\infty(\mathbb{D}, dA) $, a linear operator $S_\phi^k : L_a^2(\mathbb{D},dA) \longrightarrow L_a^2(\mathbb{D},dA)$ is defined by $$S_\phi^k(f) = W_k H_\phi (f) \text{ for all } f \in L_a^2(\mathbb{D},dA).$$ We call $S_\phi^k$ the $k^{th}$-order slant little Hankel operator on $L_a^2(\mathbb{D},dA)$ with symbol $\phi$ and

\begin{equation*}
\| S_\phi^k \| = \| W_k H_\phi \| \leq \|W_k \|\|H_\phi\| =\|H_\phi\| \leq \sqrt{k} \| \phi\|_\infty .
\end{equation*}Thus, $ S_\phi^k $ is bounded. We denote the set of all $k^{th}$-order slant little Hankel operators on $L_a^2(\mathbb{D},dA)$  by SHO($ L_a^2 $).
\end{defn}

The $ (m,n)^{th} $ entry of matrix representation of $ S_\phi ^k $ on $L_a^2(\mathbb{D},dA)$ with respect to orthonormal basis  $\{ \sqrt{n+1 } z^n \}_{n=0}^\infty $ is
\begin{align*}
\left\langle S_\phi^k \sqrt{n+1} z^n,\sqrt{m+1} z^m \right\rangle 
&= \sqrt{n+1}\sqrt{m+1} \left\langle W_k PJ( \phi z^n),z^m \right\rangle\\
&= \sqrt{n+1}\sqrt{m+1} \left\langle  PJ( \phi z^n), W_k^* z^m \right\rangle\\
&= \sqrt{n+1}\sqrt{m+1} \left\langle  PJ( \phi z^n), \frac{km+1}{m+1} z^{km} \right\rangle\\
&= \sqrt{n+1}\sqrt{m+1} \left\langle (\sum_{j=0}^\infty a_j \bar{z}^j +\sum_{j=1}^\infty b_j z^j ) z^n, \frac{km+1}{m+1} \bar{z}^{km} \right\rangle\\
&= \frac{\sqrt{n+1}(km+1)}{\sqrt{m+1}} \sum_{k=0}^\infty a_j \left\langle \bar{z}^j z^n, \bar{z}^{km} \right\rangle\\
&= \frac{\sqrt{n+1} (km+1)}{ \sqrt{m+1}(n+km+1)}a_{n+km}
\end{align*}
and its matrix representation is given as 
\begin{equation} \label{m1}
\begin{bmatrix}
a_0 & \frac{1}{\sqrt{2}}a_1 & \frac{1}{\sqrt{3}}a_2 & \frac{1}{\sqrt{4}}a_3 & \frac{1}{\sqrt{4}}a_4 & \dots \\\\
\frac{1}{ \sqrt{2}}a_{k} & \frac{(k+1)\sqrt{2}}{(k+2)\sqrt{2}}a_{k+1} & \frac{(k+1)\sqrt{3}}{(k+3)\sqrt{2}}a_{k+2} & \frac{(k+1)\sqrt{4}}{(k+4)\sqrt{2}}a_{k+3} & \frac{(k+1)\sqrt{5}}{(k+5)\sqrt{2}}a_{k+4} & \dots \\\\
\frac{1}{\sqrt{3}}a_{2k} & \frac{(2k+1)\sqrt{2}}{(2k+2)\sqrt{3}}a_{2k+1} & \frac{(2k+1)\sqrt{3}}{(2k+3)\sqrt{3}}a_{2k+2} & \frac{(2k+1)\sqrt{4}}{(2k+4)\sqrt{3}}a_{2k+3} & \frac{(2k+1)\sqrt{5}}{(2k+5)\sqrt{3}}a_{2k+4} & \dots \\\\
\vdots & \vdots & \vdots & \vdots & \vdots & \dots \\\\
\end{bmatrix}.
\end{equation}
 For $ k=2$, $ S_\phi^k $  is simply called slant little Hankel operator on $ L_a^2(\mathbb{D},dA) $. It is denoted by  $ S_\phi $ where $ W_2 $ is denoted by $ W $.

\begin{note}Since $ b_j $ does not appear in the matrix (\ref{m1}) for any natural number $ j $ therefore we have $ S_{\sum_{j=0}^\infty a_j \bar{z}^j +\sum_{j=1}^\infty b_j z^j} = S_{\sum_{j=0}^\infty a_j \bar{z}^j} $ on $L_a^2(\mathbb{D},dA)$.
\end{note}

\begin{prop}
For $ \phi_1, \phi_2 $ in $ L^{\infty}(\mathbb{D},dA) $ and $ \lambda_1, \lambda_2 \in \mathbb{C} $ then
\begin{enumerate}
\item $ S_{\lambda_1 \phi_1 + \lambda_2 \phi_2}^k = \lambda_1 S_{\phi_1}^k + \lambda_2 S_{\phi_2}^k $.
\item $ J $ is a self adjoint unitary operator on $ L^2(\mathbb{D},dA) $.

\end{enumerate}
\end{prop}
The following proposition follows directly from the above matrix with respect to the orthonormal basis.

\begin{prop} 
 The mapping $ \Gamma : L^{\infty}(\mathbb{D},dA) \rightarrow $ SHO($ L_a^2 $)  by $ \Gamma (\phi) = S_\phi ^k $ is linear but not one- one. For, if $ \phi_1 - \phi_2 \in z H_\infty(\mathbb{D},dA)$ for some $ \phi_1, \phi_2 \in L^{\infty}(\mathbb{D},dA) $ then $ S_{\phi_1} ^k = S_{\phi_2} ^k $. This gives the operator $ S_\phi ^k $ does not have unique symbol $ \phi $.
\end{prop}

In \cite{aroraweyl}, Arora and Bhola gave the point spectrum of $k^{th}$-order slant Hankel operators on Hardy space $ H^2 $ with symbol $ \bar{z}^i $ for $ i \geq 0 $. Motivated by their work, the following results are obtained.

\begin{thm} \label{thm1}
If $ \phi(z) = \sum_{i=0}^N \bar{z}^i \in L^{\infty}(\mathbb{D},dA) $ where $ N \in \{ 0,1,\cdots 2k-1 \}$, the point spectrum of $ S_\phi ^k$ is 
\begin{equation*}
 \sigma_p(S_\phi ^k) = 
 \begin{cases}
 \{ 0,1 \} & \text{if }0\leq N < k\\
 \{ 0,\lambda_1,\lambda_2 \} &  \text{if } k\leq N \leq 2k-1,
 \end{cases}
 \end{equation*}
 where $ \lambda_1 = \frac{2k+3+\sqrt{2k^2+8k+9}}{2(k+2)}$ and $ \lambda_2 = \frac{2k+3-\sqrt{2k^2+8k+9}}{2(k+2)}$.
\end{thm}

\begin{proof}
Let $ \lambda \in \sigma_p(S_\phi^k) $. Then there exists $ f \neq 0 $ in $ L_a^2(\mathbb{D},dA) $ such that $ {S_\phi^k}f = \lambda f $.
Consider $ f = \sum_{n=0}^\infty a_n z^n $ in $ L_a^2(\mathbb{D},dA) $ then $ W_k PJM_\phi(f) = \lambda f $, giving 
$ W_k PJ\left(\sum_{i=0}^N \bar{z}^i\sum_{n=0}^\infty a_n z^n \right) = \lambda \sum_{n=0}^\infty a_n z^n $ then
 $ W_k P \left(\sum_{i=0}^N\sum_{n=0}^\infty a_n z^i \bar{z}^n \right) = \lambda \sum_{n=0}^\infty a_n z^n $. Thus,
 \begin{equation}
 \label{nc1}
 W_k \left(\sum_{i=0}^N \sum_{n=0}^i  \frac{i-n+1}{(i+1)}a_n z^{i-n} \right) = \lambda \sum_{n=0}^\infty a_n z^n .
 \end{equation}
\textbf{Case 1:} If $ 0 \leq N <k $ then equation (\ref{nc1}) gives $  \sum_{i=0}^N \frac{1}{i+1} a_i = \lambda \sum_{n=0}^\infty a_n z^n $. This yields, $ \lambda a_0 = \sum_{i=0}^N \frac{1}{i+1} a_i  \text{ and } \lambda a_n = 0 \text{ for all } n \geq 1 $.\\
If $ \lambda \neq 0 $ and $ \lambda \neq 1 $ then $ a_n = 0 $ for all $ n \geq 1 $. This yields $ a_0 = \lambda a_0 $ which gives $ a_0 =0 $ leads to $ f = 0 $, a contradiction.
Hence $ 0 $ is the eigen value of $ S_\phi^k $ corresponding to the eigen vector $ f(z)=\sum_{n=0} ^\infty a_n z^n $ with $ \sum_{i=0}^N \frac{1}{i+1} a_i = 0 $ and $ 1 $ is the eigen value of $ S_\phi^k $ corresponding to the eigen vector $ f(z)= a_0 $.\\
\textbf{Case 2:}  If $ k \leq N < 2k-1 $ then equation (\ref{nc1}) gives $ \sum_{i=0}^N \frac{1}{i+1}a_i + \sum_{i=k}^N \frac{k+1}{i+1}a_{i-k}z = \lambda \sum_{n=0}^\infty a_n z^n $.
This yields 
\begin{equation}
\label{nc2} 
\sum_{i=k}^N \frac{k+1}{i+1}a_{i-k} = \lambda a_1 ,\sum_{i=0}^N \frac{1}{i+1}a_i = \lambda a_0 
\text{ and } \lambda a_n = 0  \text{ for all } n \geq 2.
\end{equation}
If $ \lambda \neq 0 $ then equation(\ref{nc2}) gives $ a_n = 0 $ for all $ n \geq 2 $, 
\begin{equation}
\label{nc3}
 a_0 +\frac{1}{2}a_1 = \lambda a_0 
\end{equation}
and
\begin{equation}
\label{nc4}
 a_0 + \frac{k+1}{k+2}a_1=\lambda a_1.
\end{equation}
On solving  equations (\ref{nc3}) and (\ref{nc4}), it follows that $ a_1 = 2(\lambda -1)a_0  $ then substituting the value of $ a_1 $ in equation (\ref{nc4}), it becomes $ a_0(\lambda - \lambda_1)(\lambda -\lambda_2) =0 $ where  $ \lambda_1 = \frac{2k+3+\sqrt{2k^2+8k+9}}{2(k+2)}$ and $ \lambda_2 = \frac{2k+3-\sqrt{2k^2+8k+9}}{2(k+2)}$.\\
If $ \lambda \neq 0, \lambda \neq \lambda_1 \text{ and } \lambda \neq \lambda_2 $ then $ a_0 =0 $ and $ a_1 =0 $ which gives $ f = 0 $, a contadiction. Hence $ 0 $ is the eigen value of $ S_\phi^k $ corresponding to the eigen vector $ f(z)=\sum_{n=2} ^\infty a_n z^n $  and $ \lambda_1 $ and $ \lambda_2 $ are the eigen values of $ S_\phi^k $ corresponding to the eigen vector $ f(z)= a_0 + a_1z $ with $ a_1 = 2(\lambda_1- 1)a_0 $ and $ a_1 = 2(\lambda_2- 1)a_0 $ respectively.
\end{proof}

\begin{rem} \label{rem1} From the proof of theorem (\ref{thm1}), for $ \phi(z) = \sum_{i=0}^N \bar{z}^i \in L^\infty(\mathbb{D},dA)$ we have the following observations:\\
\textbf{ Case 1: } If $ 0 \leq N <k $ then for any complex number $ \lambda \neq 0, 1 $ we have $ Range (S_{\phi}^k - \lambda I ) = \{  (\sum_{i=0}^N \frac{1}{i+1} a_i- \lambda a_0) - \lambda  \sum_{n=1}^\infty a_n z^n: \sum_{n=0}^\infty a_n z^n \in  L_a^2(\mathbb{D},dA) \} $ which is dense in $ L_a^2(\mathbb{D},dA) $. Therefore residual spectrum of $ S_{\phi}^k - \lambda I $, $ \sigma _R(S_{\phi}^k - \lambda I) $ is empty.\\
\textbf{ Case 2: } If $ k \leq N <2k-1 $ then for any complex number $ \lambda $ except $ \lambda = 0, \lambda_1, \lambda_2 $ we obtain that $ Range (S_{\phi}^k - \lambda I ) = \{  (\sum_{i=0}^N \frac{1}{i+1} a_i- \lambda a_0) + (\sum_{i=k}^N \frac{k+1}{i+1} a_{i-k}- \lambda a_1 ) z - \lambda  \sum_{n=2}^\infty a_n z^n: \sum_{n=0}^\infty a_n z^n \in  L_a^2(\mathbb{D},dA) \} $ is dense in $ L_a^2(\mathbb{D},dA) $. Thus  $ \sigma_R (S_{\phi}^k - \lambda I) = \Phi$ .

\end{rem}
As a consequence of the above theorem, we obtain the following result:
\begin{thm} Let $ \sigma_{AP} (S_{\phi}^k) $ and $ \sigma (S_{\phi}^k) $ denote the approximate point spectrum and spectrum of $ S_{\phi}^k $ respectively, where $ \phi(z) = \sum_{i=0}^N \bar{z}^i \in L^{\infty}(\mathbb{D},dA) $ and  $ N \in \{ 0,1,\cdots 2k-1 \}$ then $ \sigma_{AP} (S_{\phi}^k) = \sigma (S_\phi^k) $.
\end{thm}
\begin{proof} It is well known that \cite{halmoshilbert}, $ \sigma(T) = \sigma_{AP}(T) \bigcup  \sigma_{CP}(T) $ for any bounded linear operator $ T $ on Hilbert space $ H $, where $ \sigma(T), \sigma_{AP}(T) $ and $ \sigma_{CP}(T) $ denotes spectrum, approximate point spectrum and compression spectrum of $ T $ respectively. Therefore
\begin{equation} \label{nc5}
 \sigma(S_{\phi}^k) = \sigma_{AP}(S_{\phi}^k) \bigcup  \sigma_{CP}(S_{\phi}^k),
 \end{equation}
  where  $ \phi(z) = \sum_{i=0}^N \bar{z}^i \in L^{\infty}(\mathbb{D},dA) $ and  $ N \in \{ 0,1,\cdots 2k-1 \}$.
Also from \cite{halmoshilbert}, we have $ \sigma_R (S_{\phi}^k) = \sigma_{CP} (S_{\phi}^k) \backslash \sigma_P(S_{\phi}^k) $. By remark (\ref{rem1}), it is evident that $ \sigma_{CP}(S_{\phi}^k) \subseteq \sigma_P (S_{\phi}^k)$ but $ \sigma_P (S_{\phi}^k) \subseteq \sigma_{AP}(S_{\phi}^k) $. Thus, $ \sigma_{CP}(S_{\phi}^k) \subseteq \sigma_{AP}(S_{\phi}^k) $. Hence from equation \eqref{nc5}, it follows that $ \sigma_{AP} (S_{\phi}^k) = \sigma (S_\phi^k) $.
\end{proof} 

Similarly we conclude the following result:
\begin{thm} For $ i \geq 0 $, the point spectrum of $ S_{\bar{z}^i}^k$ is the following:
\begin{equation*}
 \sigma_p(S_{\bar{z}^i}^k) = 
 \begin{cases}
 \{ 0\} & \text{ if } i \text{ is not a multiple of } (k+1)\\
 \{ 0,(\frac{k+1}{k+2}) \} & \text{ if } i \text{ is a multiple of } (k+1).
 \end{cases}
 \end{equation*} and $ \sigma_{AP} (S_{\bar{z}^i}^k) = \sigma (S_{\bar{z}^i}^k) $.

\end{thm}

\section{{\bf Commutativity of $k^{th}$-order slant little Hankel operators}}
In the year $ 2013 $, C.liu and Y.Lu \cite{liuproduct} discussed the commutativity of $k^{th}$-order slant Toeplitz operators on the Bergman space with harmonic polynomial symbols, analytic symbols and coanalytic symbols. In this section, we show that under some  assumptions $k^{th}$-order slant little Hankel operators on $  L_a^2(\mathbb{D},dA) $ commute if and only if the symbols are linearly dependent.

\begin{thm}
\label{th1}Let $ \phi(z) = \sum_{i=0}^n a_i \bar{z}^i , \zeta(z) = \sum_{j=0}^n b_j \bar{z}^j $ be such that $ \phi, \zeta \in  L^\infty(\mathbb{D},dA)$, where $ n $ is any non negative integer and $ a_n \neq 0 , b_n \neq 0 $ then $ S_\phi^k $ and $ S_ \zeta^k $ commute if and only if $ \phi $ and $ \zeta $ are linearly dependent.

\end{thm}

\begin{proof} Let $ \phi $ and $ \zeta $ are linearly dependent then it is obvious that $ S_\phi^k $ and $ S_ \zeta^k $ commute. Conversely suppose that $ S_\phi^k $ and $ S_ \zeta^k $ commute.
If $ n=0 $ then result is trivially true. For $ n > 0 $  let $ n = kp+r  $ where $ p \geq 0 , 0 \leq r \leq k-1 $ be integers. Since $ S_\phi^k $ and $ S_ \zeta^k $ commute therefore, 
\begin{equation}
\label{t1}
{S_\phi^k}^* {S_\zeta^k}^* (z^p) = {S_\zeta^k}^* {S_\phi^k}^* (z^p) . 
\end{equation}
Consider
\begin{align}
{S_\phi^k}^* {S_\zeta^k}^* (z^p) =& PJM_{\hat{\phi}} W_k^* PJM_{\hat{\zeta}}W_k^*(z^p)
= PJ M_{\hat{\phi}} W_k^* PJ\left(\sum_{j=0}^n \overline{b_j} \bar{z}^j \frac{kp+1}{p+1} z^{kp}\right)\nonumber\\
=& \frac{kp+1}{p+1} PJ M_{\hat{\phi}} W_k^*P \left( \sum_{j=0}^n \overline{ b_j} z^j  \bar{z}^{kp} \right)\nonumber\\
\label{t2}
=& \frac{kp+1}{p+1} PJ M_{\hat{\phi}} W_k^* \left( \sum_{j=kp}^n \frac{(j-kp+1)}{(j+1)}\overline{b_j} z^{j-kp}\right).
\end{align}
The following two cases arise:\\
\textbf{Case 1:} If $ r= 0 $ then equation (\ref{t2}) becomes
\begin{equation}
\label{t3}
{S_\phi^k}^* {S_\zeta^k}^* (z^p) =  \frac{kp+1}{p+1} PJ M_{\hat{\phi}} W_k^* \overline{b_n}
=  \frac{kp+1}{p+1} \overline{b_n} PJ \left( \sum_{i=0}^n \overline{ a_i} \bar{z}^i \right)= \frac{kp+1}{p+1} \overline{b_n}\left( \sum_{i=0}^n \overline{ a_i} z^i \right).
\end{equation}
Similar calculation gives
\begin{equation}
\label{t4}
{S_\zeta^k}^* {S_\phi^k}^* (z^p) = \frac{kp+1}{p+1} \overline{a_n}\left( \sum_{j=0}^n \overline{ b_j} z^j \right).
\end{equation}
From equations (\ref{t1}), (\ref{t3}) and (\ref{t4}), it follows that
\begin{equation}
\label{t5} 
\frac{kp+1}{p+1} \overline{b_n}\left( \sum_{i=0}^n \overline{ a_i} z^i \right) =  \frac{kp+1}{p+1} \overline{a_n}\left( \sum_{j=0}^n \overline{ b_j} z^j \right).
\end{equation}
Since $ \{ \sqrt{n+1} z^n \}_{n=0}^\infty $ forms an orthonormal basis for Bergman space, so equation (\ref{t5}) gives $ \overline{b_n} \overline{a_i} = \overline{a_n} \overline{b_i}$ for all $ 0 \leq i \leq n $. This yields $ b_i = \lambda a_i $ for all $ 0 \leq i \leq n $ where $ \lambda = \frac{b_n}{a_n} $. Hence, $ \zeta (z) = \lambda \phi (z) $.\\
\textbf{Case 2:} If $ r > 0 $ then it follows from equation (\ref{t2}) that
\begin{align}
{S_\phi^k}^* {S_\zeta^k}^* (z^p) = & \frac{kp+1}{p+1} PJ M_{\hat{\phi}}  \left( \sum_{j=kp}^n \frac{(k(j-kp)+1)}{(j+1)}\overline{b_j}
 z^{k(j-kp)}\right)\nonumber\\
 =& \frac{kp+1}{p+1} PJ \left( \sum_{i=0}^n \overline{a_i} \bar{z}^i \sum_{j=kp}^n \frac{(k(j-kp)+1)}{(j+1)}\overline{b_j}
 z^{k(j-kp)}\right)\nonumber\\
 =& \frac{kp+1}{p+1} P \left( \sum_{i=0}^n \overline{a_i} z^i \sum_{j=kp}^n \frac{(k(j-kp)+1)}{(j+1)}\overline{b_j} \bar{z}^{k(j-kp)}\right)\nonumber\\
 =& \frac{kp+1}{p+1} P \left( \sum_{i=0}^n \overline{a_i} z^i \sum_{q=0}^r \frac{(kq+1)}{(q+kp+1)}\overline{b_{q+kp}} \bar{z}^{kq}\right)\nonumber\\
 \label{t6}
 =& \frac{kp+1}{p+1} \left( \sum_{q=0}^{min(r,p)} \sum_{i=kq}^n \frac{(kq+1)(i-kq+1)}{(q+kp+1)(i+1)} \overline{a_i} \overline{b_{q+kp}} z^{i-kq} \right).\\
 \text{Similarly, we can obtain}\nonumber\\
 \label{t7}
 {S_\zeta ^k}^* {S_\phi ^k}^* (z^p) = & \frac{kp+1}{p+1} \left( \sum_{s=0}^{min(r,p)} \sum_{j=ks}^n \frac{(ks+1)(j-ks+1)}{(s+kp+1)(j+1)} \overline{b_j} \overline{a_{s+kp}} z^{j-ks} \right).
\end{align}
Equations (\ref{t1}), (\ref{t6}) and (\ref{t7}) yield 
\begin{align}
\label{t8}
\left( \sum_{q=0}^{min(r,p)} \sum_{i=kq}^n \right.&\left.\frac{(kq+1)(i-kq+1)}{(q+kp+1)(i+1)}  \overline{a_i} \overline{b_{q+kp}} z^{i-kq} \right) =\nonumber\\
& \left( \sum_{s=0}^{min(r,p)} \sum_{j=ks}^n \frac{(ks+1)(j-ks+1)}{(s+kp+1)(j+1)} \overline{b_j} \overline{a_{s+kp}} z^{j-ks} \right).
\end{align}
Therefore for every integer $ m $ such that $ n-k < m \leq n $ , we have $ \overline{a_m} \overline{b_{kp}} = \overline{b_m} \overline{a_{kp}} $. It gives $ b_m = \lambda a_m $ where $ \lambda = \frac{b_{kp}}{a_{kp}}$. Similarly from equation (\ref{t8}) it follows that 
for every integer $ m $ such that $ n-2k < m \leq n-k $, we have $ \frac{1}{kp+1} \overline{a_m} \overline{b_{kp}} + \frac{(k+1)(m+1)}{(kp+2)(m+p+1)} \overline{a_{m+k}} \overline{b_{kp+1}} =  \frac{1}{kp+1} \overline{b_m} \overline{a_{kp}} + \frac{(k+1)(m+1)}{(kp+2)(m+p+1)} \overline{b_{m+k}} \overline{a_{kp+1}} $. Since $ n-k < m+k \leq n $ and $ r < k $ so for all $ y \geq 0 $, $ kp+y > n-k =kp + r-k $ therefore $ a_m b_{kp} = b_m a_{kp} $ implies $ b_m = \lambda a_m $. Proceeding like this, by using equation (\ref{t8}) it follows that $ b_m = \lambda a_m $ for $ 0 \leq m \leq n $ where  $ \lambda = \frac{b_{kp}}{a_{kp}} $. Hence, $ \zeta (z) = \lambda \phi (z) $.
\end{proof}

\begin{lem}Let $ \phi (z) = \sum_{i=0}^ n a_i \bar{z}^i $ and  $ \zeta (z) = \sum_{j=0}^ m b_j \bar{z}^j $ be such that $ \phi, \zeta \in  L^\infty(\mathbb{D},dA)$ where $ n $ and $ m $ are non negative integers with $ n> m $. Let $ n =kp_1 +r_1 , m = kp_2 +r_2 $ where $ p_1,p_2 ,r_1,r_2 $ are non negative integers such that $ p_1,p_2 \geq 0 $, $ 0 \leq r_1,r_2 < k $ and also let $ a_{kp_1}^2 + b_ {kp_2}^2 \neq 0 $ and $ b_m \neq 0 $. If $ S_\phi ^k $ and $ S_\zeta ^k $ commute then $ a_j = 0 $ for each integer  $ j $ such that $ m < j \leq n $.

\end{lem}

\begin{proof}
Since $ S_\phi ^k $ and $ S_\zeta ^k $ commute, therefore
\begin{equation}
\label{a1}
{S_\phi^k}^* {S_\zeta^k}^* (f) = {S_\zeta^k}^* {S_\phi^k}^* (f) \text{ for all  } f \in L_a^2 (\mathbb{D},dA). 
\end{equation}
The following three cases arise:\\
\textbf{Case 1:} If $ m =0 $ and $ n = kp_1 + r_1 $. Since $ n>m $, so either $ p_1 = 0 $ and $ 0 < r_1 $ or $ p_1 > 0 $ and $ 0 \leq r_1 < k $ then 
\begin{align}
{S_\phi^k}^* {S_\zeta^k}^* (1) = & PJ M_{\hat{\phi}} W_k^* PJ M_{\hat{\zeta}} W_k^* (1)
=   PJ M_{\hat{\phi}} W_k^* (\overline{b_m})
=   PJ M_{\hat{\phi}}  (\overline{b_m})\nonumber\\
\label{a2} 
= & \overline{b_m} PJ (\sum_{i=0}^n \overline{a_i} \bar{z}^i) = \overline{b_m}  \sum_{i=0}^n \overline{a_i} z^i, \\
{S_\zeta^k}^* {S_\phi^k}^* (1) = & PJ M_{\hat{\zeta}} W_k^* PJ M_{\hat{\phi}} W_k^* (1)
=   PJ M_{\hat{\phi}} W_k^* PJ( \sum_{i=0}^n \overline{a_i} \bar{z}^i)\nonumber\\
\label{a3}
= & PJ M_{\hat{\phi}}  ( \sum_{i=0}^n \frac{ki+1}{i+1} \overline{a_i} z^{ki}) = \overline{b_m} P ( \sum_{i=0}^n \frac{ki+1}{i+1} \overline{a_i} \bar{z}^{ki})
=   \overline{b_m}  \overline{a_0}.
\end{align}
So, from equations (\ref{a1}), (\ref{a2}) and (\ref{a3}) it follows that $ \overline{b_m} \overline{a_i} = 0 $ for  $ 0< i \leq n $. Since $ b_m \neq 0 $ therefore $ a_i =0 $ for all $ i $ such that $ m < i \leq n $.\\
\textbf{Case 2:} If $ m = r_2 $ where $ 0 < r_2 < k $.\\
If $ n = r_1 $ then since $ n > m $, so $ 0 < r_2 < r_1 <k $.
\begin{align}
{S_\phi^k}^* {S_\zeta^k}^* (1) = & PJ M_{\hat{\phi}} W_k^* PJ M_{\hat{\zeta}} W_k^* (1)
=   PJ M_{\hat{\phi}} W_k^* (\sum_{j=0}^m \overline{b_j} z^j)
=   PJ M_{\hat{\phi}}  (\sum_{j=0}^m \frac{kj+1}{j+1} \overline{b_j} z^{kj})\nonumber\\
\label{a4}
= &   PJ (\sum_{i=0}^n \overline{a_i} \bar{z}^i \sum_{j=0}^m \frac{kj+1}{j+1} \overline{b_j} z^{kj}) 
=    P (\sum_{i=0}^n \overline{a_i} z^i \sum_{j=0}^m \frac{kj+1}{j+1} \overline{b_j} \bar{z}^{kj}) 
=   \overline{b_0}  \sum_{i=0}^n \overline{a_i} z^i. \\
{S_\zeta^k}^* {S_\phi^k}^* (1) = & PJ M_{\hat{\zeta}} W_k^* PJ M_{\hat{\phi}} W_k^* (1)
=   PJ M_{\hat{\zeta}} W_k^* (\sum_{i=0}^n \overline{a_i} z^i)
=   PJ M_{\hat{\zeta}}  (\sum_{i=0}^n \frac{ki+1}{i+1} \overline{a_i} z^{ki})\nonumber\\
\label{a5}
= &   PJ (\sum_{j=0}^m \overline{b_j} \bar{z}^j \sum_{i=0}^n \frac{ki+1}{i+1} \overline{a_i} z^{ki}) 
=    P (\sum_{j=0}^m \overline{b_j} z^j \sum_{i=0}^n \frac{ki+1}{i+1} \overline{a_i} \bar{z}^{ki}) 
=  \overline{a_0}  \sum_{j=0}^m \overline{b_j} z^j.
\end{align}
From equations (\ref{a1}), (\ref{a4}) and (\ref{a5}) it follows that $ b_0 a_i = 0 $ for $ 0< i \leq n $ and $ b_0 a_j = b_j a_0 $ for $ 0 \leq  j \leq m $. In particular for $ j = m $, $ b_0 a_m = a_0 b_m $. If $ b_0 = 0 $ then $ a_0 =0 $, a contradiction as $ a_{kp_1}^2 + b_{kp_2}^2 = a_0^2 + b_0^2 \neq 0 $. Therefore $ b_0 \neq 0 $. This yields $ a_i =0 $ for all $ i $ such that $ m< i \leq n $.\\\\
If $ n = kp_1 +r_1 $ where $ p_1 >0 $ and $ 0 \leq r_1 <k $. Similar calculations gives $ {S_\phi^k}^* {S_\zeta^k}^* (z^{p_1}) = 0 $ and  $ {S_\zeta^k}^* {S_\phi^k}^* (z^{p_1}) = \frac{1}{p_1 +1} \sum_{j=0}^m \overline{b_j}  \overline{a_{kp_1}} z^j $. By using equation (\ref{a1}), it follows that $a_{kp_1} b_j =0 $ for $ 0 \leq j \leq m $. In particular for $ j=m $ we have $ a_{kp_1} b_m =0 $. Since $ b_m \neq 0 $, therefore $ a_{kp_1} =0 $. Similarly, we obtain
\begin{equation*}
{S_\phi^k}^* {S_\zeta^k}^* (1) = \sum_{j=0}^{min(r_2,p_1)} \sum_{i=kj}^n \frac{(kj+1)(i-kj+1)}{(j+1)(i+1)} \overline{a_i} \overline{b_j} z^{i-kj}
\end{equation*}
and
\begin{equation*}
{S_\zeta^k}^* {S_\phi^k}^* (1) =  \sum_{j=0}^ m  \overline{b_j} \overline{a_0} z^j.
\end{equation*}
By using equation (\ref{a1}), it follows that $ a_i b_0 =0 $ for $ n-k < i \leq n $. Since $ n-k < kp_1 \leq n $, so $ a_{kp_1} b_0 =0 $ but $ a_{kp_1} =0 $ and $ a_{kp_1} ^2 + b_{kp_2}^2 = a_{kp_1} ^2 + b_0 ^2 \neq 0 $. Therefore $ b_0 \neq 0 $. This yields $ a_i = 0 $ for $ n-k< i \leq n $. Also, from equation (\ref{a1}) we have $ a_i b_0 + \frac{k+1}{2} b_1 a_{k+i} =0 $ for all $ i $ such that $ n-2k < i \leq n-k $ but $ a_{k+i} =0 $ for $ n-k < i+k \leq n $. Hence $ a_i =0 $ for $ n-2k < i \leq n-k $. Continuing in this way, we conclude that $ a_i =0 $ for all $ i $ such that $ m < i \leq n $.\\
\textbf{Case 3:} If $ m = kp_2 $ where $ p_2 >0 $ and $ n = kp_1 +r_1 $ then since $ n > m $ so, either $ p_2 = p_1 $ and $ 0<r_1 < k  $ or $ p_1 > p_2 $ and $ 0 \leq r_1 <k $. By the simple calculations, we  obtain
\begin{equation*}
{S_\phi^k}^* {S_\zeta^k}^* (z^{p_2}) = \frac{(kp_2 +1)}{(p_2+1)(m+1)} \overline{b_m} \sum_{i=0}^n \overline{a_i} z^i 
\end{equation*}
 and 
\begin{equation*}
{S_\zeta^k}^* {S_\phi^k}^* (z^{p_2}) =\frac{(kp_2 +1)}{(p_2+1)} \sum_{q=0}^{min(n-kp_2,p_2)} \sum_{j=kq}^m \frac{(kq+1)(j-kq+1)}{(q+kp_2+1)(j+1)} \overline{b_j} \overline{a_{q+kp_2}} z^{j-kq}.
\end{equation*}
By using equation (\ref{a1}), it follows that $ b_m a_i =0 $ for $ m<i \leq n $ but $ b_m \neq 0 $ which leads to $ a_i =0 $ for all $ i $ such that $ m<i \leq n $.\\
\textbf{Case 4:} If $ m = kp_2 +r_2 $ where $ p_2 > 0 $ and $ 0 < r_2 < k $.\\
If $ n = kp_1 + r_1 $ where $ p_1 = p_2 =p $ (say) then since $ n> m $ therefore, $ 0 < r_2 < r_1 <k $. Then,
\begin{equation*}
{S_\phi^k}^* {S_\zeta^k}^* (z^{p}) = \frac{(kp +1)}{(p+1)} \sum_{q=0}^{min(r_2,p)} \sum_{i=kq}^n \frac{(kq+1)(i-kq+1)}{(q+kp+1)(i+1)} \overline{a_i} \overline{b_{q+kp}} z^{i-kq}
\end{equation*}
 and
\begin{equation*}
 {S_\zeta^k}^* {S_\phi^k}^* (z^{p}) = \frac{(kp +1)}{(p+1)} \sum_{s=0}^{min(r_1,p)} \sum_{j=ks}^m \frac{(ks+1)(j-ks+1)}{(s+kp+1)(j+1)} \overline{b_j} \overline{a_{s+kp}} z^{j-ks}.
 \end{equation*}
From equation (\ref{a1}), it follows that $ a_i b_{kp} = 0 $ for $ m < i \leq n $. Also,  $ n-k < m $ as $ r_1 < k < r_2 +k $ therefore, $ a_m b_{kp} = b_m a_{kp} $. Since $ b_m \neq 0 $, so if $ b_{kp} = 0 $ then $ a_{kp} = 0 $, a contradiction. Hence $ b_{kp} \neq 0 $ therefore $ a_i = 0 $ for each $ i $ such that $ m < i \leq n $.\\\\
If $ n = kp_1 + r_1 $ where $ p_1 > p_2 $ and $ 0 \leq r_1,r_2 <k $. Then $ {S_\phi^k}^* {S_\zeta^k}^* (z^{p_1}) = 0 $ and
\begin{equation*}
{S_\zeta^k}^* {S_\phi^k}^* (z^{p_1}) = \frac{(kp_1 +1)}{(p_1+1)} \sum_{q=0}^{min(r_1,p_2)} \sum_{j=kq}^m \frac{(kq+1)(j-kq+1)}{(q+kp_1+1)(j+1)} \overline{b_j} \overline{a_{q+kp_1}} z^{j-kq}.
\end{equation*}
From equation (\ref{a1}) it follows that $ b_j a_{kp_1} = 0 $ for $ m-k < j \leq m $. In particular for $ j =m $ we have $ b_m a_{kp_1} = 0 $. Since $ b_m \neq 0 $ therefore, $ a_{kp_1} = 0 $. Thus $ b_{kp_2} \neq 0 $ as $ a_{kp_1} ^2 + b_{kp_2} ^2 \neq 0 $. Again calculating 
\begin{equation*}
{S_\phi^k}^* {S_\zeta^k}^* (z^{p_2}) =\frac{(kp_2 +1)}{(p_2+1)} \sum_{q=0}^{min(r_2,p_1)} \sum_{i=kq}^n \frac{(kq+1)(i-kq+1)}{(q+kp_2+1)(i+1)} \overline{a_i} \overline{b_{q+kp_2}} z^{i-kq} 
\end{equation*}
 and 
\begin{equation*} {S_\zeta^k}^* {S_\phi^k}^* (z^{p_2}) = \frac{(kp_2 +1)}{(p_2+1)} \sum_{s=0}^{min(n-kp_2,p_2)} \sum_{j=ks}^m \frac{(ks+1)(j-ks+1)}{(s+kp_2+1)(j+1)} \overline{b_j} \overline{a_{q+kp_2}} z^{j-kq}
\end{equation*}
gives $ a_i b_{kp_2} = 0 $ for $ n-k <i \leq n $ (using equation (\ref{a1})). Since $ b_{kp_2} \neq 0 $ so it gives $ a_i =0 $ for $ n-k <i \leq n $. Also, $ \frac{1}{kp_2+1}a_i b_{kp_2} + \frac{(k+1)(i+1)}{(kp_2+2)(i+k+1)} a_{i+k} b_{kp_2 +1} = 0 $ for all $ i $ such that $n-2k <i \leq n-k $. Since $ a_{i+k} = 0 $ for $ n-k <i+k \leq n $  and $ b_{kp_2} \neq 0 $ leads to $ a_i =0 $ for $ n-2k <i \leq n-k $. Continuing like this, we conclude that $ a_i =0 $ for all $ i $ such that $ m <i \leq n $.
\end{proof}

\begin{thm} Let $ \phi (z) = \sum_{i=0}^ n a_i \bar{z}^i $ and $ \zeta (z) = \sum_{j=0}^ m b_j \bar{z}^j $ be such that $ \phi, \zeta \in  L^\infty(\mathbb{D},dA)$ where $n $ and $ m $ are non negative integers such that $ n> m $. Let $ n =kp_1 +r_1 , m = kp_2 +r_2 $ where $ p_1,p_2 ,r_1,r_2 $ are non negative integers such that $ p_1,p_2 \geq 0, 0 \leq r_1,r_2 < k $ and also let $  b_ {kp_2} \neq 0 $ and $ b_m \neq 0 $ then $ S_\phi ^k $ and $ S_\zeta ^k $ commute if and only if $ \phi $ and $ \zeta $ are linearly dependent.
\end{thm}

\begin{proof} Let $ \phi $ and $ \zeta $ are linearly dependent then it is obvious that $ S_\phi^k $ and $ S_ \zeta^k $ commute.
Conversely suppose that $ S_\phi^k $ and $ S_ \zeta^k $ commute. Since $  b_ {kp_2} \neq 0  $ therefore $ a_{kp_1}^2 + b_ {kp_2}^2 \neq 0 $. Hence by previous lemma, $ a_j = 0 $ for all integer  $ j $ such that $ m < j \leq n $. Let if possible, there exists  a non negative integer $ t $ with $ t \leq m $ such that $ a_i =0 $ for each $ t< i \leq m $ and $ a_t \neq 0 $. Then again by previous lemma, $ b_j = 0 $ for all integer  $ j $ such that $ t < j \leq m $ but $ b_m \neq 0 $ so, it gives $ a_m \neq 0 $. Hence by theorem (\ref{th1}), it follows that $ \phi $ and $ \zeta $ are linearly dependent.
\end{proof}

 \section{ACKNOWLEDGEMENT.} Support of CSIR-UGC Research Grant(UGC) [Ref. No.:21/12/2014(ii) EU-V, Sr. No. 2121440601] to second author for carrying out the research work is gratefully acknowledged.

\bibliographystyle{plain}

\begin{thebibliography}{99}

\bibitem{aroraslant} S. C. Arora, Ruchika Batra and M. P. Singh, {\it Slant Hankel operators}, Arch. Math. (brno), {\bf 42}(2)(2006), 125--133.

\bibitem{aroracompression} S. C. Arora and J. Bhola, {\it The Compression of a $k$th-order slant Hankel operator}, Ganita(India), {\bf 59}(1)(2008), 1--11.


\bibitem{aroraweyl} S. C. Arora and J. Bhola, {\it Weyl’s Theorem for a class of operators}, Int. J. Contemp. Math. Sciences, {\bf 6}(25)(2011), 1213--1220.



\bibitem{douglasbanach} R. G. Douglas, {\it Banach algebra techniques in operator theory}, Vol.~49, Academic  Press, New York, 1972.

\bibitem{durenbergman} P. Duren and A. Schuster, \textit{Bergman spaces}, Mathematical surveys and monographs,  Vol.~100, American mathematical soc., 2004. 

\bibitem{halmoshilbert} P. R. Halmos, \textit{A Hilbert space problem book}, Van Nostrand, Princeton, 1967. 



\bibitem{hoproperties} M. C. Ho, {\it Properties of slant Toeplitz operators}, Indiana Univ. Math.J., {\bf 45}(3)(1996), 843--862.

\bibitem{liuproduct} C. Liu and Y. Lu , {\it Product and commutativity of $k$th-order slant Toeplitz operators}, Abstr. Appl. Anal. (2013).


\bibitem{lucommutativity} Y. Lu, C. Liu and J. Yang, {\it  Commutativity of $k$th- order Slant Toeplitz Operators}, Math. Nachr., {\bf 283}(9)(2010), 1304--1313.



\bibitem{pellerexcursion} V. V. Peller, {\it An excursion into the theory of Hankel operators}, Holomorphic spaces (Berkeley, CA, 1995), Math. Sci. Res. Inst. Publ., Vol.~33
(1998), 65-–120.

\bibitem{pellerhankel} V. V. Peller, {\it Hankel operators and their applications}, Monographs in Mathematics, Springer-Verlag, New York, 2003.

\bibitem{powerhankel}  S. C. Power, {\it Hankel operators on Hilbert space}, Bull. London Math. Soc., {\bf 12}(6) (1980), 422-–442.





\end{thebibliography}

% ------------------------------------------------------------------------
\end{document}